\documentclass[12pt]{amsart}
\usepackage{a4wide,enumerate,xcolor,graphicx}
\usepackage{amsmath,bm}
\usepackage{scrextend} % for labeling environment
\allowdisplaybreaks

\let\pa\partial
\let\na\nabla
\let\eps\varepsilon
\newcommand{\N}{{\mathbb N}}
\newcommand{\R}{{\mathbb R}}

\newcommand{\diver}{\operatorname{div}}

\newcommand{\red}{\textcolor{black}}

\newtheorem{theorem}{Theorem}
\newtheorem{lemma}[theorem]{Lemma}

\newtheorem{proposition}[theorem]{Proposition}
\newtheorem{remark}[theorem]{Remark}

\begin{document}

\title[Maxwell--Stefan--Fourier systems]{Analysis of Maxwell--Stefan systems 
for heat conducting fluid mixtures}

\author[C. Helmer]{Christoph Helmer}
\address{Institute for Analysis and Scientific Computing, Vienna University of
	Technology, Wiedner Hauptstra\ss e 8--10, 1040 Wien, Austria}
\email{christoph.helmer@tuwien.ac.at}

\author[A. J\"ungel]{Ansgar J\"ungel}
\address{Institute for Analysis and Scientific Computing, Vienna University of
	Technology, Wiedner Hauptstra\ss e 8--10, 1040 Wien, Austria}
\email{juengel@tuwien.ac.at}

\date{\today}

\thanks{\red{The authors thank the anonymous reviewers for helping us to
improve significantly the initial version of the paper.}
The authors have been partially supported by the Austrian Science Fund (FWF), 
grants P30000, P33010, F65, and W1245.}

\begin{abstract}
The global-in-time existence of bounded 
weak solutions to the Maxwell--Stefan--Fourier
equations in Fick--Onsager form is proved. The model consists of the mass 
balance equations for the partial mass densities and and the energy balance
equation for the total energy. The diffusion and heat fluxes 
depend linearly on the gradients of the thermo-chemical potentials and the gradient
of the temperature and include the Soret and Dufour effects. 
The cross-diffusion system exhibits an entropy structure,
which originates from the thermodynamic modeling. The lack of positive
definiteness of the diffusion matrix is compensated by the fact that the
total mass density is constant in time. The entropy estimate yields the 
a.e.\ positivity of the partial mass densities and temperature.
\red{Also diffusion matrices are considered that degenerate for vanishing partial 
mass densities.}
\end{abstract}

% \paragraph{Keywords:}

% \paragraph{AMS classification:}

\maketitle

%%%%%%%%%%%%%%%%%%%%%%%%%%%%%%%%%%%%%%%%%%%%%%%%%%%%%%%%%%%%%%%%%%%%%%%%%%%%%%%

\section{Introduction}

Maxwell--Stefan equations describe the dynamics of multicomponent fluids
by accounting for the gradients of the chemical potentials as driving forces. 
The global existence analysis is usually
based on the so-called entropy or formal gradient-flow structure. 
Up to our knowledge, almost all existence
results are concerned with the isothermal setting.
Exceptions are the local-in-time existence result of \cite{HuSa18} and
the coupled Maxwell--Stefan and compressible Navier--Stokes--Fourier systems
analyzed in \cite{GPZ15,PiPo17}, where no temperature gradients in the
diffusion fluxes (Soret effect) have been taken into account. 
In this paper, we suggest and analyze for the first time Maxwell--Stefan--Fourier
systems in Fick--Onsager form, including Soret and Dufour effects. 

\subsection{Model equations}

We consider the evolution of the partial mass densities $\rho_i(x,t)$ and
temperature $\theta(x,t)$ in a fluid mixture, governed by the equations
\begin{align}
  & \pa_t\rho_i + \diver J_i = r_i, \quad
	J_i = -\sum_{j=1}^n M_{ij}(\bm\rho,\theta)\na q_j 
	- M_i(\bm\rho,\theta)\na\frac{1}{\theta}, \label{1.eq1} \\
  & \pa_t(\rho\theta) + \diver J_e = 0, \quad
	J_e = -\kappa(\theta)\na\theta - \sum_{j=1}^n M_j(\bm\rho,\theta)\na q_j
	\quad \mbox{in }\Omega,\ i=1,\ldots,n, \label{1.eq2}
\end{align}
where $\Omega\subset\R^3$ is a bounded domain, $\bm\rho=(\rho_1,\ldots,\rho_n)$
is the vector of mass densities, \red{$\rho=\sum_{i=1}^n\rho_i$ is the total
mass density}, and $q_i=\log(\rho_i/\theta)$ is the 
thermo-chemical potential of the $i$th species. 
The diffusion fluxes are denoted
by $J_i$, the reaction rates by $r_i$, the energy flux by $J_e$, and
the heat conductivity by $\kappa(\theta)$. 
The functions $M_{ij}$ are the diffusion coefficients, and the
terms $M_i\na(1/\theta)$ and $\sum_{j=1}^n M_j\na q_j$ describe the Soret and
Dufour effect, respectively. 

We prescribe the boundary and initial conditions
\begin{align}
  & J_i\cdot\nu=0, \quad J_e\cdot\nu + \lambda(\theta_0-\theta)=0
	\quad\mbox{on }\pa\Omega,\ t>0, \label{1.bc} \\
  & \rho_i(\cdot,0)=\rho_i^0,\quad (\rho_i\theta)(\cdot,0)=\rho_i^0\theta^0
	\quad\mbox{in }\Omega,\ i=1,\ldots,n, \label{1.ic}
\end{align}
where $\nu$ is the exterior unit normal vector to $\pa\Omega$, $\theta_0>0$ is 
the constant background temperature, and $\lambda\ge 0$ is a relaxation parameter.
Equations \eqref{1.bc} mean that the fluid cannot leave the domain $\Omega$,
while heat transfer through the boundary is possible (if $\lambda\neq 0$).

In Maxwell--Stefan systems, the driving forces $d_i$
are usually given by linear combinations of the diffusion fluxes
\cite[Sec.~14]{BoDr15}:
\begin{equation}\label{1.MSE}
  \pa_t\rho_i + \diver J_i = r_i, \quad
	d_i = -\sum_{j=1}^n b_{ij}\rho_i\rho_j\bigg(
	\frac{J_i}{\rho_i}-\frac{J_j}{\rho_j}\bigg), \quad i=1,\ldots,n,
\end{equation}
where $b_{ij}=b_{ji}\ge 0$ for $i,j=1,\ldots,n$.
\red{It is shown in \cite{BoDr20} that the Fick--Onsager and Maxwell--Stefan
formulations are equivalent, at least in the isothermal case. We show in 
Section \ref{sec.model} that \eqref{1.MSE} can be written as \eqref{1.eq1}
for a special choice of $d_i$, $M_{ij}$, and $M_i$ in the non-isothermal situation.}

We say that the diffusion fluxes in \eqref{1.eq1} are in {\em Fick--Onsager form}.
As the heat flux is given by
Fourier's law, we call system \eqref{1.eq1}--\eqref{1.eq2} the
{\em Maxwell--Stefan--Fourier} equations in Fick--Onsager form.
We refer to Section \ref{sec.model} for details of the modeling.

To fulfill mass conservation, the sum of the diffusion fluxes and the sum of
the reaction terms should vanish, i.e.\ $\sum_{i=1}^n J_i=0$ and
$\sum_{i=1}^n r_i=0$ (see Section \ref{sec.model}). 
Then, summing \eqref{1.eq1} over $i=1,\ldots,n$, we see that
the total mass density $\rho(\cdot,t)=\sum_{i=1}^n\rho_i(\cdot,t)=\rho^0$ 
is constant in time (but generally not in space).
Another consequence of the identity $\sum_{i=1}^n J_i=0$ 
is that the diffusion matrix has a nontrivial kernel, and we assume that
\begin{equation}\label{1.M1}
  \sum_{i=1}^n M_{ij}=0 \quad\mbox{for }j=1,\ldots,n, \quad \sum_{i=1}^n M_i=0.
\end{equation}
\red{For our first existence result}, 
we suppose that the matrix $(M_{ij})$ is symmetric and 
positive semidefinite in the sense that there exists $c_M>0$ such that
\begin{equation}\label{1.M2}
  \sum_{i,j=1}^n M_{ij}(\bm\rho,\theta)z_iz_j \ge c_M|\Pi \bm{z}|^2 
	\quad\mbox{for }\bm{z}\in\R^n,\ \bm\rho\in\R_+^n,\ \theta\in\R_+,
\end{equation}
where $\Pi=I-\bm{1}\otimes\bm{1}/n$ is the orthogonal projection on
$\operatorname{span}\{\bm{1}\}^\perp$. \red{This condition holds for non-dilute
fluids; we refer to Section \ref{sec.result} for a weaker condition.}

\subsection*{Notation}

We write $\bm{z}$ for a vector of $\R^n$ with components $z_1,\ldots,z_n$ and 
$\bm{z}'$ for a vector of $\R^{n-1}$ with components $z_1,\ldots,z_{n-1}$. 
In particular, $\bm{1}=(1,\ldots,1)\in\R^n$.
Furthermore, we set $\R_+=[0,\infty)$ and $\Omega_T=\Omega\times(0,T)$.

\subsection{Mathematical ideas}

The mathematical difficulties of system \eqref{1.eq1}--\eqref{1.eq2} are
the cross-diffusion structure, the lack of coerciveness of the diffusion
operator, and the temperature terms. In particular, it is not trivial
to verify the positivity of the temperature.
These difficulties are overcome by exploiting the entropy structure of the equations.
We describe the main ideas for the first existence result. 
More precisely, we use the mathematical entropy
$$
  h = \sum_{i=1}^n\rho_i(\log\rho_i-1) - \rho\log\theta.
$$
\red{Introducing the relative thermo-chemical potentials 
$v_i = \pa h/\pa\rho_i - \pa h/\pa\rho_n = q_i - q_n$ for $i = 1,\ldots,n$
and interpreting $h$ as a function of $(\bm{\rho}',\theta)$,}
a formal computation (which is made precise for an approximate scheme; see
\eqref{2.ei}) shows that
\begin{align}
  \frac{d}{dt}\int_\Omega h(\bm\rho',\theta)dx
	&+ \frac{c_M}{2}\int_\Omega\bigg(\frac{1}{n}|\na\bm v|^2 
	+ |\na\Pi\bm{q}|^2\bigg)dx \nonumber\\
	&{}+ \int_\Omega\kappa(\theta)
	|\na\log\theta|^2 dx + \lambda\int_{\pa\Omega}
	\bigg(\frac{\theta_0}{\theta}-1\bigg)ds 
	\le \sum_{i=1}^{n-1}\int_\Omega r_iv_i dx. \label{1.ei}
\end{align} 
\red{The bound for $\na\bm{v}$ comes from the positive definiteness of the
{\em reduced} diffusion matrix $(M_{ij})_{i,j=1}^{n-1}$; see Lemma \ref{lem.M}.}
Under suitable conditions on the heat conductivity and the reaction rates, 
this so-called entropy inequality provides gradient estimates for $\bm v$,
$\log\theta$, $\theta$, and $\Pi\bm{q}$, but not for the full vector $\bm{q}$. 
This problem was overcome in \cite{BJPZ20}
for a more general (but stationary) multicomponent Navier--Stokes--Fourier system
by using tools from mathematical fluid dynamics 
(effective viscous flux identity and Feireisl's oscillations defect measure).
In our model, the situation is much simpler. 
\red{Indeed, the relation $v_i=\log\rho_i-\log\rho_n$ can be inverted yielding
\begin{equation}\label{1.rhov}
  \rho_i = \frac{\rho^0\exp(v_i)}{\sum_{j=1}^n \exp(v_j)}, \quad i=1,\ldots,n-1, \quad
	\rho_n = \rho^0 - \sum_{j=1}^{n-1}\rho_j,
\end{equation}
which suggests to work with the {\em reduced} vector 
$\bm{\rho}'=(\rho_1,\ldots,\rho_{n-1})$. Moreover, 
this shows that $\rho_i$ stays bounded in some interval $(0,\rho^*)$ and, in view
of the bound for $\na \bm{v}$, that
$\na\bm\rho$ is bounded in $L^2(\Omega)$.} Together with a bound for the
(discrete) time derivative of $\rho_i$, we deduce the strong convergence
of $\rho_i$ from the Aubin--Lions compactness lemma.

Still, there remains a difficulty. 
The estimate for $\kappa(\theta)^{1/2}\na\log\theta$ in $L^2(\Omega)$ from \eqref{1.ei}
is not sufficient to define $\kappa(\theta)\na\theta$ in the weak formulation.
In the Navier--Stokes--Fourier equations, this difficulty is handled by
replacing the local energy balance by the local entropy inequality and the
global energy balance \cite{FeNo09}. We choose another approach. 
The idea is to derive better estimates for the temperature by using $\theta$
as a test function in the weak formulation of \eqref{1.eq2}.
If $\kappa(\theta)\ge c_\kappa\theta^2$ for some $c_\kappa>0$ and
$M_j/\theta$ is assumed to be bounded, then a formal computation, which is
made precise in Lemma \ref{lem.theta}, gives
\begin{align}
  \frac12\frac{d}{dt}&\int_\Omega\rho^0\theta^2 dx
	+ c_\kappa\int_\Omega\theta^2|\na\theta|^2 dx
	- \lambda\int_{\pa\Omega}(\theta_0-\theta)\theta ds \label{1.est2} \\
	&= \sum_{j=1}^{n-1} \int_\Omega \frac{M_j}{\theta}\theta\na v_j\cdot\na\theta dx
	\le \frac{c_\kappa}{2}\int_\Omega\theta^2|\na\theta|^2 dx
	+ C\sum_{j=1}^{n-1}\int_\Omega|\na v_j|^2 dx. \nonumber
\end{align}
Since $\na v_j$ is bounded in $L^2$, 
this yields uniform bounds for $\theta^2$
in $L^\infty(0,T;L^1(\Omega))$ and $L^2(0,T;H^1(\Omega))$.
These estimates are sufficient to treat the term $\kappa(\theta)\na\theta$.
The delicate point is to choose the approximate scheme in such a way that
estimates \eqref{1.ei} and \eqref{1.est2} can be made rigorous; we refer
to Section \ref{sec.ex} for details.

\subsection{State of the art}

Before we state our main result, 
we review the state of the art of Maxwell--Stefan and related models.
The isothermal equations were derived from the multi-species Boltzmann equations
in the diffusive approximation in \cite{BoBr19,BGPS13}. The Fick--Onsager form of the
Maxwell--Stefan equations was rigorously derived in Sobolev spaces
from the multi-species Boltzmann system in \cite{BrGr20}. 
The Maxwell--Stefan equations in the Fick--Onsager form, 
coupled with the momentum balance equation, 
can be identified as a rigorous second-order Chapman--Enskog approximation
of the Euler (--Korteweg) equations for multicomponent fluids; see
\cite{HJT19} for the Euler--Korteweg case and \cite{OsRo20} for the Euler case.
The work \cite{BGP19} is concerned with the
friction limit in the isothermal Euler equations using the hyperbolic
formalism developed by Chen, Levermore, and Liu.
A formal Chapman--Enskog expansion of the stationary non-isothermal model was 
presented in \cite{TaAo99}. 
Another non-isothermal Maxwell--Stefan system was derived in \cite{ABSS20}, but the
energy flux is different from the expression in \eqref{1.eq2}.

The existence analysis of (isothermal) Maxwell--Stefan equations
started with the paper \cite{GiMa98}, where the existence
of global-in-time weak solutions near the constant equilibrium was proved. 
A proof of local-in-time classical solutions
to Maxwell--Stefan systems was given in \cite{Bot11}, \red{and regularity
and instantaneous positivity for the Maxwell--Stefan system were shown in
\cite{HMPW17}}. 
In \cite{JuSt13}, the entropy or formal gradient-flow structure was revealed,
which allowed for the proof of global-in-time weak solutions with general
initial data. Maxwell--Stefan systems, coupled to the Poisson equation for 
the electric potential, were analyzed in \cite{JuLe19}. 

\red{Alt and Luckhaus \cite{AlLu83} proved a global existence result for 
parabolic systems related to the Fick--Onsager formulation. However, their
result cannot be directly applied to system \eqref{1.eq1} because of the 
lack of coerciveness. Moreover, this theory does not yield $L^\infty$ bounds. 
They are obtained from the technique of \cite{Jue15}, but the treatment of 
Soret and Dufour terms requires some care and is not contained in that work.}

All the mentioned results hold if the barycentric velocity vanishes.
For non-vanishing fluid velocities, the Maxwell--Stefan equations need to be
coupled to the momentum balance.
The Maxwell--Stefan equations were coupled to the incompressible Navier--Stokes
equations in \cite{ChJu15}, and the global existence of weak solutions was shown.
A similar result can be found in \cite{DoDo19}, where the incompressibility
condition was replaced by an artificial time derivative of the pressure and 
the limit of vanishing approximation parameters was performed. Coupled
Maxwell--Stefan and compressible Navier--Stokes equations were analyzed in
\cite{BoDr20}, and the local-in-time existence analysis was performed.
A global existence analysis for a general isothermal 
Maxwell--Stefan--Navier--Stokes system was performed in \cite{DDGG20}.
For the existence analysis of coupled stationary Maxwell--Stefan and compressible
Navier--Stokes--Fourier systems, we refer to \cite{BJPZ20,GPZ15,PiPo17}.
In \cite{BJPZ20}, temperature gradients were included in the partial mass
fluxes, but only the stationary model was investigated. The global-in-time
existence of weak solutions to the transient Maxwell--Stefan--Fourier equations
is missing in the literature and proved in this paper for the first time.

\subsection{Main results}\label{sec.result}

We impose the following assumptions:
\begin{labeling}{(A44)}
\item[(H1)] {\em Domain:} $\Omega\subset\R^3$ is a bounded domain with
a Lipschitz continuous boundary.
\item[(H2)] {\em Data:} $\theta^0\in L^\infty(\Omega)$, 
$\inf_\Omega\theta^0>0$, $\theta_0>0$,
$\lambda\ge 0$; $\rho_i^0\in \red{H^1(\Omega)} \cap L^\infty(\Omega)$ satisfies 
$0<\rho_*\le\rho_i^0\le\rho^*$ in $\Omega$ for some $\rho_*$, $\rho^*>0$.
\item[(H3)] {\em Diffusion coefficients:} For $i,j=1,\ldots,n$, the
coefficients $M_{ij}$, $M_j\in C^0(\R_+^n\times\R_+)$ satisfy
\eqref{1.M1} and $M_{ij}$, $M_i/\theta$ are bounded functions.
\item[(H4)] {\em Heat conductivity:} $\kappa \in C^0(\R_+)$ and there exist 
$c_\kappa$, $C_\kappa>0$ 
such that for all $\theta\ge 0$,
$$
  c_\kappa(1+\theta^2) \le \kappa(\theta)\le C_\kappa(1+\theta^2).
$$
\item[(H5)] {\em Reaction rates:} $r_1,\ldots,r_n\in C^0(\R^n\times\R_+)
\cap L^\infty(\R^n\times\R_+)$ satisfies
$\sum_{i=1}^n r_i=0$ and there exists $c_r>0$ such that for all $\bm{q}\in\R^n$
and $\theta>0$,
$$
  \sum_{i=1}^n r_i(\Pi\bm{q},\theta)q_i \le -c_r|\Pi\bm{q}|^2.
$$
\end{labeling}

The bounds on $\rho^0$ in Hypothesis (H2) are needed to derive the positivity
and boundedness of the partial mass densities.
In the example presented in Section \ref{sec.model},
the coefficients $M_{ij}$ and $M_i/\theta$ depend on $\rho_i$; since we prove
the existence of $L^\infty$ solutions $\rho_i$, the functions $M_{ij}$ and $M_i$ 
are indeed bounded, as required in Hypothesis (H3). The growth condition for
the heat conductivity in Hypothesis (H4) is used to derive higher integrability
of the temperature, see \eqref{1.est2}, which allows us to treat the
heat flux term. If $\lambda=0$, we can impose the weaker condition
$\kappa(\theta)\ge c_\kappa\theta^2$.
Hypothesis (H5) is satisfied for the reaction terms used in \cite{DDGG20}.
The bound for $\sum_{i=1}^n r_iq_i$  
gives a control on the $L^2(\Omega)$ norm of $\Pi\bm{q}$. 
Together with the estimates for $\na(\Pi\bm{q})$ from \eqref{1.ei}, 
we are able to infer an $H^1(\Omega)$ estimate for $\Pi\bm{q}$.  
\red{A more natural $L^2(\Omega)$ bound for $\bm{q}$ may be derived under the
assumption that the total initial density does not lie on a critical
manifold associated to the reaction rates; we refer to \cite[Theorem 11.3]{DDGG20}
for details. Vanishing reaction rates are allowed in Theorem \ref{thm.ex2} below.}

Our first main result is as follows.

\begin{theorem}[Existence]\label{thm.ex}
Let Hypotheses (H1)--(H5) hold, let $(M_{ij})$ satisfy \eqref{1.M2}, and let $T>0$. 
Then there exists a weak solution $(\bm\rho,\theta)$
to \eqref{1.eq1}--\eqref{1.ic} satisfying $\rho_i>0$, $\theta>0$ a.e.\ in $\Omega_T$, 
\begin{align}
  & \rho_i\in L^\infty(\Omega_T)\cap L^2(0,T;H^1(\Omega))\cap
	H^1(0,T;H^2(\Omega)'), \label{1.rho} \\
	& v_i\in L^2(0,T;H^1(\Omega)), \quad
	(\Pi\bm{q})_i\in L^2(0,T;H^1(\Omega)), \label{1.v} \\
	& \theta\in L^2(0,T;H^1(\Omega))\cap
	W^{1,16/15}(0,T;W^{1,16}(\Omega)'),\quad 
	\log\theta\in L^2(0,T;H^1(\Omega)); \label{1.theta}
\end{align}
where $v_i=\log(\rho_i/\rho_n)$ and $(\Pi\bm{q})_i=v_i-\sum_{j=1}^{n}v_j/n$
for $i=1,\ldots,n$; it holds that
\begin{align}
  & \int_0^T\langle\pa_t\rho_i,\phi_i\rangle dt
	+ \int_0^T\int_\Omega\bigg(\sum_{j=1}^{n-1} M_{ij}\na v_j
	- \frac{M_i}{\theta}\na\log\theta\bigg)\cdot\na\phi_i dxdt
	= \int_0^T\int_\Omega r_i\phi_i dxdt, \label{1.weak1} \\
	& \int_0^T\langle \pa_t(\rho\theta),\phi_0\rangle dt
	+ \int_0^T\int_\Omega \kappa(\theta)\na\theta\cdot\na\phi_0 dxdt
	+ \int_0^T\int_\Omega\sum_{j=1}^{n-1} M_j\na v_j\cdot\na\phi_0 dxdt \label{1.weak2} \\
	&\phantom{xx}{}
	= \lambda\int_0^T\int_{\pa\Omega}(\theta_0-\theta)\phi_0 dxds \nonumber
\end{align}
for all $\phi_1,\ldots,\phi_n\in L^2(0,T;H^1(\Omega))$, 
$\phi_0\in L^\infty(0,T;W^{1,\infty}(\Omega))$ with $\na\phi_0\cdot\nu=0$
on $\pa\Omega$, and $i=1,\ldots,n$; and the initial conditions
\eqref{1.ic} are satisfied in the sense of $H^2(\Omega)'$ and $W^{1,16}(\Omega)'$,
respectively.
\end{theorem}

\red{The weak formulation can be written in various variable sets since
\begin{align*}
  \sum_{j=1}^{n-1} M_{ij}\na v_j 
	&= \sum_{j=1}^n M_{ij}\na(\Pi\bm{q})_j
	= \sum_{j=1}^n M_{ij}\na q_j, \\
	\sum_{j=1}^{n-1}M_j\na v_j 
	&= \sum_{j=1}^n M_j\na(\Pi\bm{q})_j
	= \sum_{j=1}^n M_j\na q_j,
\end{align*}
whenever the corresponding variables are defined. Thus,
our definition of a weak solution is compatible with \eqref{1.eq1}--\eqref{1.eq2}.}
The proof is based on a suitable approximate scheme, uniform bounds coming
from entropy estimates, and $H^1(\Omega)$ estimates for the partial mass densities.
More precisely, we use two levels of approximations. First, we replace the time
derivative by an implicit Euler discretization to overcome issues with the
time regularity. Second, we add higher-order regularizations for the
thermo-chemical potentials and the logarithm of the temperature $w=\log\theta$
to achieve $H^2(\Omega)$ regularity for these variables. 
Since we are working in three space dimensions, we conclude $L^\infty(\Omega)$ 
solutions, which are needed to define properly $\rho_i=\exp(w+q_i)$.

A priori estimates are deduced from a discrete version of the entropy 
inequality \eqref{1.ei}. They are derived from
the weak formulation by using $v_i$ and $e^{-w_0}-e^{-w}$ as test functions,
where $w_0=\log\theta_0$. The
entropy structure is only preserved if we add additionally a $W^{1,4}(\Omega)$
regularization and some lower-order regularization in $w$. 
The properties for the heat conductivity allow us to obtain estimates 
for $\theta$ in $H^1(\Omega)$ and for $\na\log\theta$ in $L^2(\Omega)$.
Property \eqref{1.M2} provides gradient estimates for $\bm{v}$ and, in view of
\eqref{1.rhov}, also for $\bm\rho$.

\red{Condition \eqref{1.M2} provides a control on the relative
thermo-chemical potentials $v_i$, but it excludes the dilute limit, 
i.e.\ situations when the mass densities vanish.
This situation is included in the recent work \cite{Dru20}, which deals with
the isothermal case. We are able to replace 
condition \eqref{1.M2} by a degenerate one, which allows for dilute mixtures:
\begin{equation}\label{1.M3}
  \sum_{i,j=1}^n M_{ij}(\bm\rho,\theta)z_iz_j \ge c_M\sum_{i=1}^n\rho_i(\Pi\bm{z})_i^2
	\quad\mbox{for }\bm{z}\in\R^n,\ \bm{\rho}\in\R_+^n,\ \theta\in\R_+.
\end{equation}
This corresponds to ``degenerate'' diffusion coefficients $M_{ij}$; see Section
\ref{sec.model} for a motivation. Although this
hypothesis seems to complicate the problem, there are two advantages. First,
it allows us to derive a gradient bound for $\rho_i^{1/2}$, and second, it
helps us to avoid the bound from $r_i$ in Hypothesis (H5). In fact, we may
assume that $r_i=0$.
\begin{theorem}[Existence, ``degenerate'' case]\label{thm.ex2}
Let condition \eqref{1.M3} be satisfied. Moreover, let
Hypotheses (H1)--(H4) hold for $T>0$ and additionally,
$(\rho_i^0)^{1/2}\in H^1(\Omega)\cap L^\infty(\Omega)$, $M_{ij}/\rho_j$ and
$M_j/\rho_j$ are bounded, $r_i=0$ for all $i,j=1,\ldots,n$.
Then there exists a weak solution $(\bm\rho,\theta)$
to \eqref{1.eq1}--\eqref{1.ic} satisfying $\rho_i\ge 0$, $\theta>0$ a.e.\ in $\Omega_T$,
\eqref{1.rho}, \eqref{1.theta}, 
and the weak formulation \eqref{1.weak1}--\eqref{1.weak2} with, respectively,
\begin{align*}
  & \sum_{i=1}^n \frac{M_{ij}}{\rho_j}\na\rho_j, \quad
	\sum_{i=1}^n \frac{M_i}{\rho_i}\na\rho_i
	\quad\mbox{instead of}\quad
	\sum_{i=1}^{n-1}M_{ij}\na v_j, \quad \sum_{i=1}^{n-1}M_i\na v_i.
\end{align*}
\end{theorem}}

The paper is organized as follows. We explain the thermodynamical modeling
of \eqref{1.eq1}--\eqref{1.eq2} in Section \ref{sec.model} and
show that the Maxwell--Stefan formulation \eqref{1.MSE} for specific $d_i$
can be written as \eqref{1.eq1} for certain coefficients $M_{ij}$ and $M_i$. 
Theorems \ref{thm.ex} and \ref{thm.ex2} are proved in Sections \ref{sec.ex}
and \ref{sec.degen}, respectively.

%%%%%%%%%%%%%%%%%%%%%%%%%%%%%%%%%%%%%%%%%%%%%%%%%%%%%%%%%%%%%%%%%%%%%%%%%%%%%%

\section{Modeling}\label{sec.model}

We consider an ideal fluid mixture consisting of $n$ components with the
same molar \red{masses} in a fixed container $\Omega\subset\R^3$. 
The balance equations for the partial mass densities $\rho_i$ are given by
$$
  \pa_t\rho_i + \diver(\rho_i v_i) = r_i, \quad i=1,\ldots,n,
$$
where $v_i$ are the partial velocities and $r_i$ the reaction rates. 
Introducing the total mass density
$\rho=\sum_{i=1}^n\rho_i$, the barycentric velocity $v=\rho^{-1}\sum_{i=1}^n\rho_iv_i$,
and the diffusion fluxes $J_i=\rho_i(v_i-v)$, we can reformulate
the mass balances as
\begin{equation}\label{2.mass}
  \pa_t\rho_i + \diver(\rho_i v+J_i) = r_i, \quad i=1,\ldots,n.
\end{equation}
By definition, we have $\sum_{i=1}^nJ_i=0$, which means that the total
mass density satisfies $\pa_t\rho+\diver(\rho v)=0$.
We assume that the barycentric velocity vanishes, $v=0$, i.e., the barycenter of the
fluid is not moving. Consequently, the total mass density is constant in time.

The non-isothermal dynamics of the mixture is assumed to be 
given by the balance equations
$$
  \pa_t\rho_i + \diver J_i = r_i, \quad \pa_t E + \diver J_e = 0, \quad
	i=1,\ldots,n,
$$
where $J_e$ is the energy flux and $E$ the total energy. 
We suppose that the diffusion fluxes are 
proportional to the gradients of the thermo-chemical potentials $q_j$ 
and the temperature gradient (Soret effect) and that
the energy flux is linear in the temperature gradient and the
gradients of $q_j$ (Dufour effect): 
$$
  J_i = -\sum_{j=1}^n M_{ij}\na q_j - M_i\na\frac{1}{\theta}, \quad i=1,\ldots,n,
	\quad J_e = -\kappa(\theta)\na\theta - \sum_{j=1}^n M_j\na q_j.
$$
The proportionality factor $\kappa(\theta)$ 
between the heat flux and the temperature gradient is the heat (or thermal) 
conductivity. 

The thermo-chemical potentials and the total energy are determined in a 
thermodynamically consistent way from the free energy
$$
  \psi(\bm\rho,\theta) = \theta\sum_{i=1}^n \rho_i(\log\rho_i-1) 
	- \rho\theta(\log\theta-1).
$$
For simplicity, we have set the heat capacity equal to one.
The physical entropy $s$, the chemical potentials $\mu_i$, and the total energy $E$
are defined by the free energy according to
\begin{align*}
  s &= -\frac{\pa \psi}{\pa\theta} = -\sum_{i=1}^n\rho_i(\log\rho_i-1) 
	+ \rho\log\theta,\\
	\mu_i &= \red{\frac{\pa \psi}{\pa\rho_i} = \theta(\log(\rho_i/\theta)+1)}, 
	\quad i=1,\ldots,n, \\
  E &= \psi + \theta s = \rho\theta.
\end{align*}
We introduce the mathematical entropy $h:=-s$ and
the thermo-chemical potentials $q_j=\mu_j/\theta=\log(\rho_j/\theta)+1$ 
for $j=1,\ldots,n$. These definitions lead to system \eqref{1.eq1}--\eqref{1.eq2}.
The Gibbs--Duhem relation yields the pressure
\red{$p=-\psi+\sum_{i=1}^n\rho_i\mu_i=\rho\theta$} of an ideal gas mixture. Note that
we do not need a pressure blow-up at $\rho=0$ to exclude vacuum or
a superlinear growth in $\theta$ to control the temperature.
\red{Note also that, because of the nonvanishing pressure, one may criticize the 
choice of vanishing barycentric velocity. In the general case, the mass and
energy balances need to be coupled to the momentum balance for $v$. Such systems,
but only for isothermal or stationary systems, have been analyzed in, e.g.,
\cite{BJPZ20,ChJu15,DDGG20,Dru16}. The choice $v=0$ is a mathematical simplification.}

If the molar masses $m_i$ of the components are not the same, we need to modify
the free energy according to \cite[Remark 1.2]{BJPZ20}
$$
  \psi = \theta\sum_{i=1}^n\frac{\rho_i}{m_i}\bigg(\log\frac{\rho_i}{m_i}-1\bigg)
	- c_W\rho\theta(\log\theta-1),
$$
where $c_W>0$ is the heat capacity. For simplicity, we have set
$m_i=1$ and $c_W=1$.

We show that the Maxwell--Stefan equations 
\begin{equation}\label{2.MSE}
  \pa_t\rho_i + \diver J_i = r_i, \quad
	d_i = -\sum_{j=1}^n b_{ij}\rho_i\rho_j\bigg(
	\frac{J_i}{\rho_i}-\frac{J_j}{\rho_j}\bigg), \quad i=1,\ldots,n,
\end{equation}
with $b_{ij}=b_{ji}>0$
can be formulated as \eqref{1.eq1} for a specific choice of $d_i$, $M_{ij}$, and $M_i$.
\red{The coefficients $b_{ij}$ may be interpreted as friction coefficients and
can depend on $(\bm\rho,\theta)$; see \cite[Section 4]{BoDr20}.
The equivalence between the Fick--Onsager and Maxwell-Stefan formulations was
thoroughly investigated in \cite{BoDr20}, and we adapt their proof to
our non-isothermal framework. For this, we introduce the matrix $B=(B_{ij})$ satisfying
$B_{ii}=\sum_{j=1,\,j\neq i}^n b_{ij}\rho_j$ and
$B_{ij}=-b_{ij}\rho_i $ for $j\neq i$. It is not invertible since $\bm\rho\in
\operatorname{ker}(B)$, but its group inverse $B^\#$ exists uniquely, satisfying 
$BB^\#=B^\#B=I-(\bm{\rho}/\rho)\otimes\bm{1}$ and
\begin{equation}\label{2.Bsharp}
  \sum_{j=1}^n B_{ij}^\#\rho_j = 0, \quad \sum_{j=1}^n B_{ji}^\# = 0\quad
	\mbox{for }i=1,\ldots,n.
\end{equation}
Furthermore, we introduce the projection $P=(P_{ij})=I-\bm{1}\otimes(\bm{\rho}/\rho)$ on
$\operatorname{span}\{\bm{\rho}\}^\perp$. 
}

\begin{proposition}\label{prop}
Define the driving forces
\red{\begin{equation}\label{2.di}
  d_i = \rho_i\na\frac{\mu_i}{\theta} - \frac{\rho_i}{\rho\theta}\na(\rho\theta)
	- 2\rho_i\theta\na\frac{1}{\theta} + q_i\rho_i\na\log\theta \quad 
	\mbox{for }i=1,\ldots,n,
\end{equation}
where the numbers $q_i\in\R$ satisfy $\sum_{i=1}^n q_i\rho_i=0$.}
Then \eqref{1.MSE} can be written as \eqref{1.eq1} with 
\begin{equation}\label{M}
  M_{ij} = \sum_{k=1}^n B_{ik}^\#\rho_k P_{kj}, \quad 
	M_i = -\theta\sum_{k=1}^n B_{ik}^\#\rho_k q_k
	\quad\mbox{for } i,j=1,\ldots,n,
\end{equation}
where $(M_{ij})$ is symmetric and 
$M_{ij}$ and $M_i$ satisfy \eqref{1.M1}. 
\end{proposition}

\red{The first three terms in the driving forces \eqref{2.di} are the same as
\cite[(4.18)]{BoDr20} and \cite[(2.11)]{BoDr15}, 
while the last term is motivated from \cite[(A5)]{TaAo99}.
A computation shows that $\sum_{i=1}^n d_i = 0$ which is consistent with \eqref{2.MSE}.
It is argued in \cite{BoDr20} that $M_{ij}$ is of the form 
$\rho_i(a_i(\bm{\rho},\theta)\delta_{ij}+\rho_j S_{ij}(\bm{\rho},\theta))$
for some functions $a_i$ and $S_{ij}$, and in the nondegenerate case, one may assume
that $a_i(\bm\rho,\theta)$ stays positive
when $\bm\rho\to\widetilde{\bm{\rho}}$ with $\widetilde\rho_i=0$ \cite[(6.6)]{BoDr20}.
This formulation motivates condition \eqref{1.M3}.}

\red{\begin{proof}
The proof is based on the equivalence between the Fick--Onsager and Maxwell--Stefan
formulations elaborated in \cite[Section 4]{BoDr20} for the isothermal case.
First, the driving forces can be formulated as
$$
  d_i = \rho_i\na\frac{\mu_i}{\theta} - \rho_i\na\log\frac{\rho}{\theta}
	+ q_i\rho_i\na\log\theta,
$$
which shows that
$$
  \sum_{j=1}^n\rho_j\na\frac{\mu_j}{\theta}
	= \sum_{j=1}^n\bigg(d_j + \rho_j\na\log\frac{\rho}{\theta}
	- q_j\rho_j\na\log\theta\bigg) = \rho\na\log\frac{\rho}{\theta}.
$$
Consequently, another formulation is
$$
  d_i = \rho_i\na\frac{\mu_i}{\theta} 
	- \frac{\rho_i}{\rho}\sum_{j=1}^n\rho_j\na\frac{\mu_j}{\theta}
  + q_i\rho_i\na\log\theta = \sum_{j=1}^n\rho_i P_{ij}\na\frac{\mu_j}{\theta}
	+ q_i\rho_i\na\log\theta.
$$
Setting $R=\operatorname{diag}(\rho_1,\ldots,\rho_n)$ and
$\bm{q}^*=\operatorname{diag}(q_1\rho_1,\ldots,q_n\rho_n)$, we obtain
$\bm{d} = R P\na(\bm{\mu}/\theta) + \bm{q}^*\na\log\theta$.
On the other hand, by \eqref{2.MSE},
$$
  d_i = -\bigg(\sum_{j=1,\,j\neq i}^n b_{ij}\rho_j\bigg)J_i 
	+ \sum_{j=1,\,j\neq i}^n b_{ij}\rho_i J_j = -\sum_{j=1}^n B_{ij}J_j.
$$
This shows that $\bm{d}=-B\bm{J}$ and hence
$\bm{J}=-B^\#\bm{d} = -B^\# RP\na(\bm{\mu}/\theta) - B^\# \bm{q}^*\na\log\theta$.
Thus, defining $M_{ij}$ and $M_i$ as in \eqref{M}, it follows that
$$
  J_i = -\sum_{j=1}^n M_{ij}\na\frac{\mu_j}{\theta} - M_i\na\frac{1}{\theta}.
$$
The matrix $\tau=BR$ is symmetric and so does $\tau^\#$. Moreover, 
by \cite[(4.26)]{BoDr20}, $B^\# = P^\top R\tau^\#P^\top$. 
Therefore, $M = B^\# RP = P^\top R\tau^\#RP$ is symmetric. We deduce from
the properties \eqref{2.Bsharp} that
$$
  \sum_{j=1}^n M_{ij} = \sum_{j,k=1}^n B^\#_{ik}\rho_k\bigg(\delta_{kj}
	- \frac{\rho_j}{\rho}\bigg) = 0,
	\quad \sum_{i=1}^n M_i = -\theta\sum_{j=1}^n\bigg(\sum_{i=1}^n B^\#_{ij}\bigg)
	\rho_j q_j = 0.
$$
This finishes the proof.
\end{proof}}

%%%%%%%%%%%%%%%%%%%%%%%%%%%%%%%%%%%%%%%%%%%%%%%%%%%%%%%%%%%%%%%%%%%%%%%%%%%%%%

\section{Proof of Theorem \ref{thm.ex}}\label{sec.ex}

The idea of the proof is to reformulate equations \eqref{1.eq1}--\eqref{1.eq2}
in terms of the relative potentials $v_i$, to approximate the resulting equations
by an implicit Euler scheme, and to add some higher-order regularizations in space
for the variables $v_i$ and $w=\log\theta$.
The de-regularization limit is based on the compactness coming from the
entropy estimates and an estimate for the temperature.

Set $w_0=\log\theta_0$, $\eps>0$, $N\in\N$, and $\tau=T/N>0$. 
\red{To simplify the notation, we set $\bm{v}=(\bm{v}',0)=(v_1,\ldots,v_{n-1},0)$ and
$\bar{\bm{v}}=(\bar{v}_1,\ldots,\bar{v}_{n-1},0)$.
Let $(\bar{\bm{v}},\bar{w})\in L^\infty(\Omega;\R^{n+1})$ 
be given, and set $\rho_i(\bm{v})=\rho^0 e^{v_i}/\sum_{j=1}^{n}e^{v_j}$
for $i=1,\ldots,n-1$, $\rho_n=\rho^0-\sum_{i=1}^{n-1}\rho_i$, and
$q_i=\log\rho_i-w$ for $i=1,\ldots,n$.} We define the approximate scheme
\begin{align}
  0 &= \frac{1}{\tau}\int_\Omega(\rho_i(\bm v)-\bar\rho_i(\bar{\bm v}))\phi_i dx
	+ \int_\Omega\bigg(\sum_{j=1}^{n-1} M_{ij}(\bm\rho,e^w)\na v_j
	- M_i(\bm\rho,e^w)e^{-w}\na w\bigg)\cdot\na\phi_i dx \label{2.app1} \\
	&\phantom{xx}{}+ \eps\int_\Omega\big(D^2 v_i:D^2\phi_i + v_i\phi_i\big)dx
	- \int_\Omega r_i(\red{\Pi\bm{q}},e^w)\phi_i dx, \nonumber \\
  0 &= \frac{1}{\tau}\int_\Omega(E-\bar E)\phi_0 dx
	+ \int_\Omega\bigg(\kappa(\theta)\na\theta + \sum_{j=1}^{n-1} M_j(\bm\rho,e^w)
	\na v_j\bigg)\cdot\na\phi_0 dx \label{2.app2} \\
	&\phantom{xx}{}- \lambda\int_{\pa\Omega}(\theta_0-\theta)\phi_0 ds
	+ \eps\int_\Omega e^w\big(D^2 w:D^2\phi_0 + |\na w|^2\na w\cdot\na\phi_0\big)dx 
	\nonumber \\
	&\phantom{xx}{}+ \red{\eps\int_\Omega(e^{w_0}+e^w)(w-w_0)\phi_0 dx} \nonumber 
\end{align}
for test functions $\phi_i\in H^2(\Omega)$, $i=0,\ldots,n-1$. 
Here, $D^2 u$ is the Hessian matrix of the function $u$, ``:'' denotes
the Frobenius matrix product, and $E=\rho^0\theta$, $\bar E=\rho^0\bar\theta$.
The lower-order regularization $\eps(e^{w_0}+e^w)(w-w_0)$ yields an $L^2(\Omega)$
estimate for $w$. Furthermore,
the higher-order regularization guarantees that $v_i$, $w\in H^2(\Omega)\hookrightarrow
L^\infty(\Omega)$, while the $W^{1,4}(\Omega)$ regularization term for $w$ allows
us to estimate the higher-order terms when using the test function 
$e^{-w_0}-e^{-w}$.

{\em Step 1: solution of the linearized approximate problem.}
In order to define the fixed-point operator, we need to solve a linearized problem.
To this end, let $y^*=(\bm{v}^*,w^*)\in W^{1,4}(\Omega;\R^{n})$ and $\sigma\in[0,1]$ 
be given. We want to find the unique solution
$y=(\bm{v}',w)\in H^2(\Omega;\R^{n})$ to the linear problem
\begin{equation}\label{LM}
  a(y,\phi) = \sigma F(\phi) \quad\mbox{for all }
	\phi=(\phi_0,\ldots,\phi_{n-1})
	\in H^2(\Omega;\R^{n}),
\end{equation}
where
\begin{align*}
  a(y,\phi) &= \int_\Omega\sum_{i,j=1}^{n-1} M_{ij}(\bm{\rho}^*,e^{w^*})\na v_j
	\cdot\na\phi_i dx + \int_\Omega\kappa(e^{w^*})e^{w^*}\na w\cdot\na\phi_0 dx \\
	&\phantom{xx}{}+ \eps\int_\Omega\sum_{i=1}^{n-1}
	\big(D^2 v_i:D^2\phi_i + v_i\phi_i\big)	\\
	&\phantom{xx}{}+ \eps\int_\Omega e^{w^*}\big(D^2 w:D^2\phi_0 
	+ |\na w^*|^2\na w\cdot\na\phi_0\big)
	+ \red{\eps\int_\Omega (e^{w_0}+e^{w^*})w \phi_0 dx}, \\
  F(\phi) &= -\frac{1}{\tau}\int_\Omega\sum_{i=1}^{n-1}(\rho^*_i-\bar\rho_i)\phi_i dx
	- \frac{1}{\tau}\int_\Omega(E^*-\bar E)\phi_0 dx 
	+ \lambda\int_{\pa\Omega}(e^{w_0}-e^{w^*})\phi_0 dx \\
	&\phantom{xx}{}
	+ \int_\Omega\sum_{i=1}^{n-1} M_i(\bm{\rho}^*,e^{w^*})e^{-w^*}\na w^*\cdot\na\phi_i dx
	- \int_\Omega\sum_{j=1}^{n-1} M_j(\bm{\rho}^*,e^{w^*})\na v_j^*\cdot\na\phi_0 dx \\
	&\phantom{xx}{}+ \int_\Omega\sum_{i=1}^n r_i(\red{\Pi\bm{q}^*},e^{w^*})\phi_i dx
	+ \red{\eps \int_\Omega(e^{w_0}+e^{w^*}) w_0 \phi_0 dx}
\end{align*}
and $\rho_i^*=\rho_i(\bm{v}^*)$, $\rho^*=\sum_{i=1}^n\rho_i^*$, $E^*=\rho^0 e^{w^*}$. 
By Hypothesis (H3) and the generalized Poincar\'e inequality 
\cite[Chap.~2, Sec.~1.4]{Tem97}, we have
$$
  a(y,y) \ge \eps\int_\Omega\big(|D^2\bm{v}|^2 + |\bm{v}|^2\big)dx
	+ \eps\int_\Omega\red{e^{w^*}}(|D^2 w|^2 + w^2)dx 
	\ge \eps C(\|\bm{v}\|_{H^2(\Omega)}^2 + \|w\|_{H^2(\Omega)}^2).
$$
Thus, $a$ is coercive. Moreover, $a$ and $F$ are continuous on $H^2(\Omega;\R^{n})$.
The Lax--Milgram lemma shows that \eqref{LM} possesses a unique solution
$(\bm{v}',w)\in H^2(\Omega;\R^{n})$.

{\em Step 2: solution of the approximate problem.}
The previous step shows that the fixed-point operator
$S:W^{1,4}(\Omega;\R^{n})\times[0,1]\to W^{1,4}(\Omega;\R^{n})$,
$S(y^*,\sigma)=y$, where $y=(\bm{v}',w)$ solves \eqref{LM},
is well defined. It holds that $S(y,0)=0$, $S$ is continuous, and since
$S$ maps to $H^2(\Omega;\R^{n})$, which is compactly embedded into
$W^{1,4}(\Omega;\R^{n})$, it is also compact. It remains to determine
a uniform bound for all fixed points $y$ of $S(\cdot,\sigma)$, where $\sigma\in(0,1]$. 
Let $y$ be such a fixed point. Then
$y\in H^2(\Omega;\R^{n})$ solves \eqref{LM} with $(\bm{v}^*,w^*)$
replaced by $y=(\bm{v}',w)$. With the test functions $\phi_i=v_i$ for 
$i=1,\ldots,n-1$
and $\phi_0=e^{-w_0}-e^{-w}$ (we need this test function since $\phi_0=-e^{-w}$ does
not allow us to control the lower-order term), we obtain
\begin{align*}
  0 &= \frac{\sigma}{\tau}\int_\Omega\sum_{i=1}^{n-1}
	(\rho_i(\bm v)-\rho_i(\bar{\bm v}))v_i dx
	+ \frac{\sigma}{\tau}\int_\Omega(E-\bar E)(-e^{-w})dx
	+ \frac{\sigma}{\tau}\int_\Omega(E-\bar E)e^{-w_0}dx \\
	&\phantom{xx}{}+ \int_\Omega\sum_{i,j=1}^{n-1} M_{ij}\na v_i\cdot\na v_j dx
	+ \int_\Omega\kappa(e^w)e^w\na w\cdot\na(-e^{-w})dx 
	- \sigma\int_\Omega\sum_{i=1}^{n-1} r_iv_i dx \\
	&\phantom{xx}{}- \sigma\int_\Omega\sum_{j=1}^{n-1} M_j e^{-w}\na w\cdot\na v_j dx 
	+ \sigma\int_\Omega\sum_{j=1}^{n-1} M_j\na v_j\cdot\na(-e^{-w})dx \\
  &\phantom{xx}{}- \sigma\lambda\int_{\pa\Omega}(e^{w_0}-e^{w})(e^{-w_0}-e^{-w})dx
	+ \eps\int_\Omega\sum_{i=1}^{n-1}\big(|D^2 v_i|^2 + v_i^2\big)dx \\
	&\phantom{xx}{}+ \eps\int_\Omega(e^{w_0}+e^w)(w-w_0)(e^{-w_0}-e^{-w})dx \\
	&\phantom{xx}{}+ \eps\int_\Omega\big(|D^2 w|^2 \red{-} D^2w:\na w\otimes\na w 
	+ |\na w|^4\big)dx \\
	&=: I_1+\cdots+ I_{12}.
\end{align*}
We see immediately that $I_7+I_8=0$. Furthermore,
$$
  I_1+I_2 = \frac{\sigma}{\tau}\int_\Omega\bigg(\sum_{i=1}^{n-1}(\rho_i-\bar\rho_i)
	\frac{\pa h}{\pa\rho_i} + (\theta-\bar\theta)\frac{\pa h}{\pa\theta}\bigg)dx.
$$
The function $(\bm \rho',\theta) \mapsto h(\bm \rho',\theta)
= \sum_{i=1}^n\rho_i(\log\rho_i-1)-\rho^0\log\theta$ with
$\rho_n=\rho^0-\sum_{i=1}^{n-1}\rho_i$ is convex, 
since the second derivatives are given by 
$$
 	\frac{\partial^2 h}{\partial \rho_i^2} = \frac{1}{\rho_i}+\frac{1}{\rho_n}, \quad
	\frac{\partial^2 h}{\partial \theta^2} = \frac{\rho^0}{\theta^2}, \quad
	\frac{\partial^2 h}{\partial \rho_i \partial \theta} = 0, \quad 
	\frac{\partial^2 h}{\partial \rho_i \partial \rho_j} = \frac{1}{\rho_n},
$$
hence we can conclude in the same way as in \cite{JuSt13}
that the Hessian is positive definite by Sylvester's criterion.
This shows that
$$ 
  h(\bm\rho',\theta)-h(\bar{\bm\rho}',\bar\theta)
	\le \sum_{i=1}^{n-1}\frac{\pa h}{\pa\rho_i}(\bm\rho',\theta)(\rho_i-\bar\rho_i)
	+ \frac{\pa h}{\pa\theta}(\bm\rho',\theta)(\theta-\bar\theta)
$$
and consequently,
$$ 
  I_1 + I_2 \ge \frac{\sigma}{\tau}\int_\Omega\big(h(\bm\rho',\theta)
	- h(\bar{\bm\rho}',\bar\theta)\big)dx.
$$

\red{For the estimate of $I_4$, we need the following lemma.
\begin{lemma}\label{lem.M}
Let the matrix $(M_{ij})\in\R^{n\times n}$ satisfy \eqref{1.M1} and \eqref{1.M2}. Then
$$
  \sum_{i,j=1}^{n-1}M_{ij}(z_i-z_n)(z_j-z_n)
	\ge \frac{c_M}{n}\sum_{i=1}^{n-1}|z_i-z_n|^2.
$$
\end{lemma}
\begin{proof}
We use \eqref{1.M1} and then \eqref{1.M2} to find for any $\bm{z}\in\R^n$ that
\begin{equation}\label{3.MM}
  \sum_{i,j=1}^{n-1}M_{ij}(z_i-z_n)(z_j-z_n)
	= \sum_{i,j=1}^n M_{ij}z_iz_j \ge c_M|\Pi\bm{z}|^2.
\end{equation}
By Jensen's inequality, we have $(n-1)\sum_{i=1}^{n-1}z_i^2\ge(\sum_{i=1}^{n-1}z_i)^2$,
which is equivalent to $n|\Pi\bm{z}|^2\ge \sum_{i=1}^{n-1}(z_i-z_n)^2$. Inserting
this inequality into \eqref{3.MM} finishes the proof.
\end{proof}
By Lemma \ref{lem.M} and Hypothesis (H5),
\begin{align*}
	I_4 &= \frac12\int_\Omega\bigg(\sum_{i,j=1}^{n}M_{ij}\na q_i\cdot\na q_j
	+ \sum_{i,j=1}^{n-1}M_{ij}\na v_i\cdot\na v_j\bigg) dx \\
	&\ge \frac{c_M}{2}\int_\Omega|\na\Pi\bm{q}|^2dx 
	+ \frac{c_M}{2n}\int_\Omega|\na\bm{v}|^2 dx, \\
	I_6 &= \sigma\int_\Omega\sum_{i=1}^n r_i q_i dx 
	\ge \sigma c_r\int_\Omega|\Pi\bm{q}|^2 dx.
\end{align*}}
Next, we have
\begin{align*}
  I_5 &= \int_\Omega\kappa(e^w)|\na w|^2 dx, \quad
	I_9 = 2\sigma\lambda\int_{\pa\Omega}(\cosh(w_0-w)-1)ds \ge 0, \\
	I_{11} &= 2\eps\int_\Omega \left(w-w_0\right)\sinh(w-w_0)dx 
	\ge \eps\int_\Omega \left(w-w_0\right)^2 dx, \\
	I_{12} &= \frac{\eps}{2}\int_\Omega\big(|D^2w|^2 
	+ |D^2w-\na w\otimes\na w|^2 + |\na w|^4\big)dx.
\end{align*}
Summarizing these estimates and applying the generalized Poincar\'e inequality, 
we arrive at the {\em discrete entropy inequality}
\begin{align}
  \frac{\sigma}{\tau}\int_\Omega & \big(h(\bm{\rho}',\theta) + e^{-w_0}E\big)dx
	+ \frac{c_M}{2}\int_\Omega\bigg(\frac{1}{n}|\na\bm{v}|^2 + |\na\Pi\bm{q}|^2
	+ \sigma c_r|\Pi\bm{q}|^2\bigg)dx \nonumber \\
	&\phantom{xx}{}+ \eps C\big(\|\bm v\|_{H^2(\Omega)}^2 + \|w\|_{H^2(\Omega)}^2
	+ \|w\|_{W^{1,4}(\Omega)}^4\big) + \int_\Omega\kappa(e^w)|\na w|^2 dx \nonumber \\
	&\le \frac{\sigma}{\tau}\int_\Omega\big(h(\bar{\bm\rho}',\bar\theta)
	+ e^{-w_0}\bar E\big)dx +  2\eps\|w_0\|_{L^2(\Omega)}^2. \label{2.ei}
\end{align}
\red{We observe that the left-hand side is bounded from below since
$-\rho^0\log\theta+e^{-w_0}E=\rho^0(-\log\theta+\theta/\theta_0)$ is
bounded from below.
The bound for $\Pi\bm{q}$ implies an $L^2(\Omega)$ bound for $\bm{v}$
since $|\bm{v}|^2\le n|\Pi\bm{q}|^2$; see the proof of Lemma \ref{lem.M}.}

Estimate \eqref{1.ei} gives a uniform bound for $(\bm v',w)$ in $H^2(\Omega;\R^{n})$ 
and consequently also in $W^{1,4}(\Omega;$ $\R^{n})$, which proves the claim.
We infer from the Leray--Schauder fixed-point theorem that there exists a solution
$(\bm v',w)$ to \eqref{2.app1}--\eqref{2.app2}.

{\em Step 3: temperature estimate.} We need a better estimate for the temperature.

\begin{lemma}\label{lem.theta}
Let $(\bm\rho,w)$ be a solution to \eqref{2.app1}--\eqref{2.app2} and set
$\theta=e^w$. Then there exists a constant $C>0$ independent of $\eps$ and
$\tau$ such that
$$
  \frac{1}{\tau}\int_\Omega\rho^0\theta^2 dx 
	+ \frac12\int_\Omega\kappa(\theta)|\na\theta|^2 dx
	\le C + \frac{1}{\tau}\int_\Omega\rho^0\bar{\theta}^2 dx
	+ C\int_\Omega|\na\bm v|^2 dx.
$$
\end{lemma}

\begin{proof}
We use the test function $\theta$ in \eqref{2.app2}.
Observing that $(E-\bar E)\theta=\rho^0(\theta-\bar\theta)\theta
\ge(\rho^0/2)(\theta^2-\bar{\theta}^2)$ and that $\kappa(\theta)
\ge c_\kappa(1+\theta^2)$ by Hypothesis (H4), we find that
\begin{align}
  \frac{1}{2\tau}&\int_\Omega\rho^0(\theta^2-\bar{\theta}^2)dx
	+ \frac12\int_\Omega\kappa(\theta)|\na\theta|^2 dx
	+ \frac{c_\kappa}{2}\int_\Omega\theta^2|\na\theta|^2 dx
	- \lambda\int_{\pa\Omega}(\theta_0-\theta)\theta dx \nonumber \\
	&\le -\sum_{j=1}^{n-1}\int_\Omega M_j\na v_j\cdot\na\theta dx
	- \eps\int_\Omega\theta(D^2\log\theta:D^2\theta 
	+ |\na\log\theta|^2\na\log\theta\cdot\na\theta)dx \nonumber \\
  &\phantom{xx}{}- \eps\int_\Omega(\theta_0+\theta)(\log\theta-\log\theta_0)\theta dx
	\nonumber \\
	&=: J_1+J_2+J_3. \label{2.aux}
\end{align}
Since $M_j/\theta$ is assumed to be bounded,
$$
  J_1 \le \frac{c_\kappa}{2}\int_\Omega\theta^2|\na\theta|^2 dx
	+ C\sum_{j=1}^{n-1}\int_\Omega|\na v_j|^2 dx.
$$
Furthermore,
\begin{align*}
  J_2 &= -\eps\int_\Omega\bigg(
	-\frac{1}{\theta}\na\theta\cdot D^2\theta\na\theta
	+ |D^2\theta|^2	+ \frac{1}{\theta^2}|\na\theta|^4\bigg)dx \\
  &= -\frac{\eps}{2}\int_\Omega\bigg(|D^2\theta|^2 
	+ \frac{1}{\theta^2}|\na\theta|^4
	+ \bigg|D^2\theta-\frac{1}{\theta}
	\na\theta\otimes\na\theta\bigg|^2\bigg) dx \le 0.
\end{align*}
The last integral $J_3$ is bounded since $-\theta^2\log\theta$ is the dominant term.
The last term on the left-hand side of \eqref{2.aux} is bounded from below by
$-(\lambda/2)\int_{\pa\Omega}\theta_0^2 dx$, which finishes the proof.
\end{proof}

\begin{remark}\rm
Better estimates can be derived if we assume that
$\kappa(\theta)\ge c_\kappa(1+\theta^{\alpha+1})$ for $\alpha\in(1,2)$.
Indeed, using $\theta^\alpha$ as a test function in \eqref{2.app2}, we find that
\begin{align}
  \frac{1}{\tau}&\int_\Omega\rho^0(\theta-\bar\theta)\theta^{\alpha}dx
	+ \alpha c_\kappa\int_\Omega\theta^{2\alpha}|\na\theta|^2 dx
  - \lambda\int_{\pa\Omega}(\theta_0-\theta)\theta^\alpha dx \nonumber \\
	&\le -\alpha\sum_{j=1}^{n-1}\int_\Omega M_j\theta^{\alpha-1}\na v_j\cdot\na\theta dx 
	- \eps\int_\Omega(\theta_0 + \theta)(\log \theta-\log \theta_0)\theta^\alpha
	 dx \nonumber \\
	&\phantom{xx}{}- \eps\int_\Omega\theta\big(D^2\log\theta:D^2\theta^\alpha
	+ |\na\log\theta|^2\na\log\theta\cdot\na\theta^\alpha\big) dx \nonumber \\
	&=: J_4+J_5+J_6. \label{2.rem}
\end{align}
A tedious but straightforward computation shows that $J_6\ge 0$ if $\alpha\in(1,2)$.
Furthermore, since $M_j/\theta$ is bounded,
$$
  J_4 \le \frac{\alpha c_\kappa}{2}\int_\Omega\theta^{2\alpha}|\na\theta|^2 dx
	+ C\sum_{j=1}^{n-1}\int_\Omega|\na v_j|^2 dx.
$$
The first integral on the right-hand side is controlled by the left-hand side
of \eqref{2.rem}. This yields a bound for $\theta^{\alpha+1}\in 
L^\infty(0,T;L^1(\Omega))\cap L^2(0,T;H^1(\Omega))\subset L^{8/3}(\Omega_T)$
(see Lemma \ref{lem.theta3}) and consequently
$\theta\in L^{8(\alpha+1)/3}(\Omega_T)$, which is better than the result
in Lemma \ref{lem.theta3}.
\qed
\end{remark}

{\em Step 4: uniform estimates.} Let $((\bm{v}')^k,w^k)$ be a solution to 
\eqref{2.app1}--\eqref{2.app2} for given 
$(\bm{v}')^{k-1}=\bar{\bm{v}}'$ and
$w^{k-1}=\bar{w}$, where $k\in\N$. 
We set 
$$  
  \theta^k=\exp(w^k), \quad 
	\rho_i^k=\exp(w^k+q_i^k) = \frac{\rho^0 e^{v_i^k}}{\sum_{j=1}^n e^{v_j^k}}
$$ 
for $i=1,\ldots,n-1$, and $E^k=\rho^0\theta^k$.
We introduce piecewise constant functions in time. For this, let
$\rho_i^{(\tau)}(x,t)=\rho_i^k(x)$, $\theta^{(\tau)}(x,t)=\theta^k(x)$, 
$v_i^{(\tau)}(x,t)=v_i^k(x)$, $q_i^{(\tau)}=\log(\rho_i^{(\tau)}/\theta^{(\tau)})$,
and $E^{(\tau)}(x,t)=E^k(x)$ 
for $x\in\Omega$, $t\in((k-1)\tau,k\tau]$, $k=1,\ldots,N$. 
At time $t=0$, we set $\rho_i^{(\tau)}(x,0)=\rho_i^0(x)$ and
$\theta^{(\tau)}(x,0)=\theta^0(x)$ for $x\in\Omega$. Furthermore, we
introduce the shift operator $(\sigma_\tau\rho_i^{(\tau)})(x,t)=\rho_i^{k-1}(x)$
for $x\in\Omega$, $t\in((k-1)\tau,k\tau]$. 
Let $(\bm\rho')^{(\tau)}=(\rho^{(\tau)}_1,\ldots,\rho^{(\tau)}_{n-1})$. 
Then $((\bm\rho')^{(\tau)},\theta^{(\tau)})$
solves (see \eqref{2.app1}--\eqref{2.app2})
\begin{align}
  0 &= \frac{1}{\tau}\int_0^T\int_\Omega(\rho^{(\tau)}_i-\sigma_\tau\rho_i^{(\tau)})
	\phi_i dxdt \label{2.app3} \\
	&\phantom{xx}{}+ \int_0^T\int_\Omega\bigg(\sum_{j=1}^{n-1}
	M_{ij}(\bm\rho^{(\tau)},\theta^{(\tau)})\na v_j^{(\tau)}
	+ M_i(\bm\rho^{(\tau)},\theta^{(\tau)})\na\frac{1}{\theta^{(\tau)}}\bigg)
	\cdot\na\phi_i dxdt \nonumber \\
	&\phantom{xx}{}+ \eps\int_0^T\int_\Omega\big(D^2 v_i^{(\tau)}:D^2\phi_i
	+ v_i^{(\tau)}\phi_i\big)dxdt 
	- \int_0^T\int_\Omega r_i(\Pi\bm q^{(\tau)},\theta^{(\tau)})
	\phi_i dxdt, \nonumber \\
	0 &= \frac{1}{\tau}\int_0^T\int_\Omega(E^{(\tau)}-\sigma_\tau E^{(\tau)})\phi_0 dxdt
	- \lambda\int_0^T\int_{\pa\Omega}(\theta_0-\theta^{(\tau)})\phi_0dsdt 
	\label{2.app4} \\
	&\phantom{xx}{}+ \int_0^T\int_\Omega\bigg(\kappa(\theta^{(\tau)})\na\theta^{(\tau)}
	+ \sum_{j=1}^{n-1} M_j(\bm\rho^{(\tau)},\theta^{(\tau)})\na v_j^{(\tau)}\bigg)
	\cdot\na\phi_0 dxdt \nonumber \\
	&\phantom{xx}{}+ \eps\int_0^T\int_\Omega
	\theta^{(\tau)}\big(D^2\log\theta^{(\tau)}:D^2\phi_0
	+ |\na\log\theta^{(\tau)}|^2\na\log\theta^{(\tau)}\cdot\na\phi_0\big)dxdt 
	\nonumber \\
	&\phantom{xx}{}+ \eps\int_0^T\int_\Omega(\theta_0 + \theta^{(\tau)})
	(\log\theta^{(\tau)}-\log \theta_0 )\phi_0 dxdt. \nonumber 
\end{align}

The discrete entropy inequality \eqref{2.ei} and the $L^\infty$ bound for 
$\rho_i^{(\tau)}$ imply the following uniform bounds:
\begin{align*}
  \|\rho_i^{(\tau)}\|_{L^\infty(0,T;L^\infty(\Omega))}
	+ \|\theta^{(\tau)}\|_{L^\infty(0,T;L^1(\Omega))}
  &\le C, \\
	\|v_i^{(\tau)}\|_{L^2(0,T;H^1(\Omega))}
	+ \|\kappa(\theta^{(\tau)})^{1/2}\na\log\theta^{(\tau)}\|_{L^2(\Omega_T)} 
	&\le C, \\
	\eps^{1/2}\|v_i^{(\tau)}\|_{L^2(0,T;H^2(\Omega))}
	+ \eps^{1/2}\|\log\theta^{(\tau)}\|_{L^2(0,T;H^2(\Omega))} &\le C, \\
	\eps^{1/4}\|\log\theta^{(\tau)}\|_{L^4(0,T;W^{1,4}(\Omega))} &\le C,
\end{align*}
for all $i=1,\ldots,n-1$, where $C>0$ is independent of $\eps$ and $\tau$. 
Hypothesis (H4) yields
\begin{equation}\label{3.nath0}
  \|\na\theta^{(\tau)}\|_{L^2(\Omega_T)} 
	+ \|\na\log\theta^{(\tau)}\|_{L^2(\Omega_T)} \le C.
\end{equation}

\begin{lemma}[Estimates for the temperature]
There exists a constant $C>0$ which does not depend on $\eps$ or $\tau$ such that
\begin{equation}\label{3.nath}
  \|\theta^{(\tau)}\|_{L^2(0,T;H^1(\Omega))}
	+ \|\log\theta^{(\tau)}\|_{L^2(0,T;H^1(\Omega))} \le C.
\end{equation}
\end{lemma}

\begin{proof}
The entropy inequality shows that $-\log\theta^{(\tau)}+\theta^{(\tau)}$ is 
uniformly bounded from above, which shows that $|\log\theta^{(\tau)}|$ is uniformly
bounded too and hence, $\log\theta^{(\tau)}$ is bounded in $L^\infty(0,T;L^1(\Omega))$.
Together with the $L^\infty(0,T;L^1(\Omega))$ bound for $\theta^{(\tau)}$, 
estimate \eqref{3.nath0}, and the 
Poincar\'e--Wirtinger inequality, we find that
\begin{align*}
  \|\theta^{(\tau)}\|_{L^2(\Omega_T)}
	&\le C\|\theta^{(\tau)}\|_{L^2(0,T;L^1(\Omega))}
	+ \|\na\theta^{(\tau)}\|_{L^2(\Omega_T)} \le C, \\
 \|\log\theta^{(\tau)}\|_{L^2(\Omega_T)}
	&\le C\|\log\theta^{(\tau)}\|_{L^2(0,T;L^1(\Omega))}
	+ \|\na\log\theta^{(\tau)}\|_{L^2(\Omega_T)} \le C,
\end{align*}
from which we conclude the proof.
\end{proof}

\red{We proceed by proving more uniform estimates. Because of the $L^2(\Omega_T)$
bound of $\na v_i^{(\tau)}$ and
\begin{align}
  \int_0^T\int_\Omega|\na\rho_i^{(\tau)}|^2dx dt
	&= \int_0^T\int_\Omega|\na\rho^0|^2\bigg|\frac{\exp(v_i^{(\tau)})}{\sum_{j=1}^n
	\exp(v_j^{(\tau)})}\bigg|^2dxdt \nonumber \\
	&\phantom{xx}{}+ \int_0^T\int_\Omega\bigg|\frac{\exp(v_i^{(\tau)})\na v_i^{(\tau)}}{
	\sum_{j=1}^n \exp(v_j^{(\tau)})}
  - \frac{\exp(v_i^{(\tau)})\sum_{j=1}^n\exp(v_j^{(\tau)})\na v_j^{(\tau)}}{
	(\sum_{j=1}^n \exp(v_j^{(\tau)}))^2}\bigg|^2 dxdt \nonumber \\
	&\le \int_0^T\int_\Omega|\na\rho^0|^2 dxdt
	+ 2\int_0^T\int_\Omega|\na\bm{v}|^2dxdt\le C, \label{est.rho}
\end{align}
$(\na\rho_i^{(\tau)})$ is bounded in $L^2(\Omega_T)$ and, taking into account the
$L^\infty$ bound for $\rho_i^{(\tau)}$, the family $(\rho_i^{(\tau)})$
is bounded in $L^2(0,T;H^1(\Omega))$. By Lemma \ref{lem.theta} and
Hypothesis (H4), $(\na(\theta^{(\tau)})^2)$ is bounded in $L^2(\Omega_T)$.
Therefore, since $((\theta^{(\tau)})^2)$ is bounded in $L^1(\Omega_T)$,
the Poincar\'e--Wirtinger inequality gives a uniform bound for $(\theta^{(\tau)})^2$
in $L^2(0,T;H^1(\Omega))$. These bounds yields higher integrability of 
$\theta^{(\tau)}$, as shown in the following lemma.}

\begin{lemma}\label{lem.theta3}
There exists $C>0$ independent of $\eps$ and $\tau$ such that
$(\theta^{(\tau)})$ is bounded in $L^{16/3}(\Omega_T)$.
\end{lemma}

\begin{proof}
We deduce from the bound for $(\theta^{(\tau)})^2$ in $L^2(0,T;H^1(\Omega))
\subset L^2(0,T;L^6(\Omega))$ that $(\theta^{(\tau)})$ is bounded in 
$L^4(0,T;L^{12}(\Omega))$. By interpolation
with $1/r=\alpha/12 + (1-\alpha)/2$ and $r\alpha=4$,
\begin{align*}
  \|\theta^{(\tau)}\|_{L^r(\Omega_T)}^r
	&= \int_0^T\|\theta^{(\tau)}\|_{L^r(\Omega)}^r dt
	\le \int_0^T\|\theta^{(\tau)}\|_{L^{12}(\Omega)}^{r\alpha}
	\|\theta^{(\tau)}\|_{L^2(\Omega))}^{r(1-\alpha)}dt \\
	&\le \|\theta^{(\tau)}\|_{L^\infty(0,T;L^2(\Omega))}^{r(1-\alpha)}\int_0^T
	\|\theta^{(\tau)}\|_{L^{12}(\Omega)}^4 dt \le C.
\end{align*}
The solution of $1/r=\alpha/12 + (1-\alpha)/2$ and $r\alpha=4$ is 
$\alpha=3/4$ and $r=16/3$.
\end{proof}

\begin{lemma}
There exists $C>0$ independent of $\eps$ and $\tau$ such that
\begin{equation}\label{est.time}
  \tau^{-1}\|\rho_i^{(\tau)}-\sigma_\tau\rho_i^{(\tau)}\|_{L^2(0,T;H^2(\Omega)')}
	+ \tau^{-1}\|\theta^{(\tau)}-\sigma_\tau\theta^{(\tau)}
	\|_{L^{16/15}(0,T;W^{1,16}(\Omega)')} \le C.
\end{equation}
\end{lemma}

\begin{proof}
Let $\phi_0 \in L^{16}(0,T;W^{1,16}(\Omega))$, $\phi_1,\ldots,\phi_{n-1}\in 
L^2(0,T;H^2(\Omega))$ and set $M_i^{(\tau)}=M_i(\bm\rho^{(\tau)},$ $\theta^{(\tau)})$,
$r_i^{(\tau)}=r_i(\bm\rho^{(\tau)},\theta^{(\tau)})$ for $i=1,\ldots,n-1$.
It follows from \eqref{2.app3}--\eqref{2.app4} and Hypotheses (H3)--(H5) that
\begin{align*}
	&\frac{1}{\tau}\bigg|\int_0^T\int_\Omega(\rho_i^{(\tau)}-\sigma_\tau\rho_i^{(\tau)})
	\phi_i dx dt\bigg|
	\le C\|\na\bm{v}^{(\tau)}\|_{L^2(\Omega_T)}\|\na\bm{\phi}\|_{L^2(\Omega_T)} \\
	&\phantom{xx}{}+ \sum_{i=1}^{n-1}\|M_i^{(\tau)}/\theta^{(\tau)}\|_{L^\infty(\Omega_T)}
	\|\na\log\theta^{(\tau)}\|_{L^2(\Omega_T)}\|\na\bm\phi\|_{L^2(\Omega_T)} \\
	&\phantom{xx}{}
	+ \eps\|\bm{v}^{(\tau)}\|_{L^2(0,T;H^2(\Omega))}
	\|\bm\phi\|_{L^2(0,T;H^2(\Omega))}
	+ \|\bm{r}^{(\tau)}\|_{L^2(\Omega_T)}\|\bm\phi\|_{L^2(\Omega_T)} \\
	&\le C\|\bm\phi\|_{L^2(0,T;H^2(\Omega))},
\end{align*}
and
\begin{align*}
	&\frac{1}{\tau}\bigg|\int_0^T\int_\Omega(E^{(\tau)}-\sigma_\tau E^{(\tau)})
	\phi_0 dxdt\bigg| \\
	&\le C + C\|\theta^{(\tau)}\|_{L^{8/3}(\Omega_T)}
	\|\na(\theta^{(\tau)})^2\|_{L^2(\Omega_T)}\|\na\phi_0\|_{L^8(\Omega_T)} \\
	&\phantom{xx}{}+ \sum_{j=1}^{n-1}\|M_j^{(\tau)}/\theta^{(\tau)}\|_{L^\infty(\Omega_T)}
	\|\theta^{(\tau)}\|_{L^{8/3}(\Omega_T)}
	\|\na v_j^{(\tau)}\|_{L^2(\Omega_T)}\|\na\phi_0\|_{L^8(\Omega_T)} \\
	&\phantom{xx}{}+ \lambda\|\theta_0-\theta^{(\tau)}\|_{L^{8/7}(0,T;L^{8/7}(\pa\Omega))}
	\|\phi_0\|_{L^8(0,T;L^8(\pa\Omega))} \\
	&\phantom{xx}{}
	+ \eps \| \theta^{(\tau)} \|_{L^3(\Omega_T)} \|\log\theta^{(\tau)}
	\|_{L^2(0,T;H^2(\Omega))}\| \nabla \phi_0\|_{L^6(\Omega_T)} \\
	&\phantom{xx}{} + \eps
	\|\theta^{(\tau)}\|_{L^{16/3}(\Omega_T)}\|\na\log\theta^{(\tau)}\|_{L^4(\Omega_T)}^3
	\|\na\phi_0\|_{L^{16}(\Omega_T)} \\
	&\phantom{xx}{}
	+ \eps C\big(1+\|\theta^{(\tau)}\log\theta^{(\tau)}\|_{L^2(\Omega_T)}\big)
	\|\phi_0\|_{L^2(\Omega_T)}
	\le C\|\phi_0\|_{L^{16}(0,T;W^{1,16}(\Omega))}.
\end{align*}
Since $|E^{(\tau)}-\sigma_\tau E^{(\tau)}|=\rho^0|\theta^{(\tau)}
-\sigma_\tau\theta^{(\tau)}|
\ge\rho_*|\theta^{(\tau)}-\sigma_\tau\theta^{(\tau)}|$, this concludes the proof.
\end{proof}

{\em Step 5: limit $(\eps,\tau)\to 0$.}
Estimates \eqref{est.rho}--\eqref{est.time} allow us to
apply the Aubin--Lions lemma in the version of \cite{DrJu12}. Thus, there
exist subsequences that are not relabeled such that as $(\eps,\tau)\to 0$,
\begin{equation}\label{strong}
  \rho_i^{(\tau)}\to\rho_i, \quad \theta^{(\tau)}\to\theta
	\quad\mbox{strongly in }L^2(\Omega_T),\ i=1,\ldots,n-1.
\end{equation}
The $L^\infty(\Omega_T)$ bound for $(\rho_i^{(\tau)})$ and the $L^{16/3}(\Omega_T)$
bound for $(\theta^{(\tau)})$ imply the stronger convergences
\begin{align*}
  \rho_i^{(\tau)}\to\rho_i &\quad\mbox{strongly in }L^r(\Omega_T)
	\mbox{ for all }r<\infty, \\
	\theta^{(\tau)}\to\theta &\quad\mbox{strongly in }L^\eta(\Omega_T)
	\mbox{ for all }\eta < 16/3.
\end{align*}
The uniform bounds also imply that, up to subsequences, 
\begin{align*}
  \rho_i^{(\tau)}\rightharpoonup \rho_i 
	&\quad\mbox{weakly in }L^2(0,T;H^1(\Omega)),\\
  \theta^{(\tau)}\rightharpoonup \theta 
	&\quad\mbox{weakly in }L^2(0,T;H^1(\Omega)),\\
	\nabla v_i^{(\tau)}\rightharpoonup \nabla v_i
	&\quad\mbox{weakly in }L^2(0,T;L^2(\Omega)),\\
	\tau^{-1}(\rho_i^{(\tau)}-\sigma_\tau\rho_i^{(\tau)})\rightharpoonup \pa_t\rho_i
	&\quad\mbox{weakly in }L^2(0,T;H^2(\Omega)'), \\
	\tau^{-1}(\theta^{(\tau)}-\sigma_\tau\theta^{(\tau)})\rightharpoonup \pa_t \theta
	&\quad\mbox{weakly in }L^{16/15}(0,T;W^{2,16}(\Omega)'),
\end{align*}
where $i = 1,\ldots,n-1$ and $j = 1,\ldots,n$.
Moreover, as $(\eps,\tau)\to 0$,
$$
  \eps \log\theta^{(\tau)} \to 0, \quad \eps v_i^{(\tau)}\to 0 \quad\mbox{strongly in }
	L^2(0,T;H^2(\Omega)).
$$
\red{At this point, $v_i$ is any limit function; we prove below that 
$v_i=\log(\rho_i/\rho_n)$.}

We deduce from the linearity and boundedness of the trace operator 
$H^1(\Omega) \hookrightarrow H^{1/2}(\pa\Omega)$ that
$$
	\theta^{(\tau)} \rightharpoonup \theta \quad\mbox{weakly in }
	L^2(0,T;H^{1/2}(\pa \Omega)).
$$
Using the compact embedding $H^{1/2}(\pa \Omega) \hookrightarrow 
L^2(\pa \Omega)$, this gives 
$$
	\theta^{(\tau)} \to \theta \quad\mbox{strongly in } 
	L^2(0,T;L^2(\pa \Omega)).
$$
\red{The a.e.\ convergence of $\rho_i$ for $i = 1,\ldots,n-1$ implies that, up to
a subsequence,
\begin{align*}
	\rho_n^{(\tau)} = \rho^0 - \sum_{i=1}^{n-1} \rho_i^{(\tau)}
	\rightarrow \rho^0 - \sum_{i=1}^{n-1} \rho_i =: \rho_n
	\quad\mbox{a.e. in }\Omega_T.
\end{align*}}

\red{Next, we prove that $\theta$ and $\rho_i$ are positive a.e. We know already
that $\theta^{(\tau)}$ and $\rho_i^{(\tau)}$ are positive in $\Omega_T$. 
It follows from the $L^\infty(0,T;L^1(\Omega))$ bound for $\log\theta^{(\tau)}$ 
and the a.e.\ pointwise convergence $\theta^{(\tau)}\to\theta$ that 
$\log\theta$ is finite a.e.\ and therefore $\theta>0$ a.e.\ in $\Omega_T$.
For the positivity of $\rho_i$, we observe first that there exists a constant
$C(n)>0$ such that for all $z_1,\ldots,z_{n-1}\in\R$, 
$$
  \log\bigg(1+\sum_{i=1}^{n-1}e^{z_i}\bigg) \le C(n)\bigg(1+\sum_{i=1}^{n-1}|z_i|\bigg).
$$
Since $\rho_i^{(\tau)}=\rho^0\exp(v_i^{(\tau)})/\sum_{j=1}^n\exp(v_j^{(\tau)})$, 
$\rho^0\ge\rho_*$, and $v_i^{(\tau)}$ is bounded in $L^1(\Omega)$, this implies for 
sufficiently small $\delta>0$ that
\begin{align*}
  \operatorname{meas}&\big\{(x,t):\rho_i^{(\tau)}(x,t)\le\delta\big\}
	= \operatorname{meas}\bigg\{(x,t):-\log\frac{\rho^0(x)\exp(v_i^{(\tau)}(x,t))}{
	\sum_{j=1}^n \exp(v_j^{(\tau)}(x,t))}\ge-\log\delta\bigg\} \\
	&\le \operatorname{meas}\bigg\{(x,t):\sum_{j=1}^n|v_j^{(\tau)}(x,t)|
	\ge C(1 - \log\delta + \log\rho_*)\bigg\} \\
	&\le \frac{C}{-\log\delta}\int_0^T\int_\Omega\sum_{i=1}^n|v_i^{(\tau)}(x,t)|dxdt
	\le \frac{C}{-\log\delta}, \quad i=1,\ldots,n-1.
\end{align*}
We infer from
\begin{align*}
	\mathrm{meas}&\big\{\liminf_{(\eps,\tau) \to 0}
	\{ (x,t):\rho_i^{(\tau)}(x,t)\le\delta\} \big\} 
	\leq \liminf_{(\eps,\tau) \to 0}\mathrm{meas} 
	\{ (x,t):\rho_i^{(\tau)}(x,t)\le\delta\} \\
	&\le \limsup_{(\eps,\tau) \to 0}\mathrm{meas} 
	\{ (x,t):\rho_i^{(\tau)}(x,t)\le\delta\}
	\le \mathrm{meas} \big\{\limsup_{(\eps,\tau) \to 0}
	\{ (x,t):\rho_i^{(\tau)}(x,t)\le\delta\} \big\},
\end{align*}
and the pointwise convergence $\rho_i^{(\tau)}\to\rho_i$ that in fact equality
holds in the previous chain of inequalities, which means that
$$
  \operatorname{meas}\{(x,t):\rho_i(x,t)\le\delta\}
	= \lim_{(\eps,\tau)\to 0}\operatorname{meas}\{(x,t):\rho_i^{(\tau)}(x,t)\le \delta\}
	\le \frac{C}{-\log\delta}
$$
and $\rho_i>0$ a.e.\ in the limit $\delta\to 0$, where $i=1,\ldots,n-1$.
We prove in a similar way for $\rho_n^{(\tau)}
=\rho^0/(\sum_{j=1}^n\exp(v_i^{(\tau)}))>0$ that $\rho_n>0$ a.e.}

\red{As $\rho_i^{(\tau)}$ converges a.e.\ to an a.e.\ positive limit, we have
$$
  v_i^{(\tau)} = \log\rho_i^{(\tau)}-\log\rho_n^{(\tau)}
	\to \log\rho_i-\log\rho_n \quad\mbox{a.e. in }\Omega_T.
$$
Thus $v_i=\log\rho_i-\log\rho_n$. Furthermore, $q_i^{(\tau)}=\log\rho_i^{(\tau)}
- \log\theta^{(\tau)} \to \log\rho_i - \log\theta =: q_i$ and
$$
  (\Pi\bm{q}^{(\tau)})_i = v_i^{(\tau)} - \frac{1}{n}\sum_{j=1}^{n}v_j^{(\tau)}
	\to v_i - \frac{1}{n}\sum_{j=1}^{n}v_j =: U_i \quad\mbox{a.e. in }\Omega_T.
$$
This shows that $v_i=q_i-q_n$ and $U_i=(q_i-q_n)-\sum_{j=1}^{n}(q_j-q_n)/n
=(\Pi\bm{q})_i$. The a.e.\ convergence of
$(\Pi\bm{q}^{(\tau)})$ and the boundedness of $r_i$ by Hypothesis (H5) lead to
$$
  r_i(\Pi\bm{q}^{(\tau)},\theta^{(\tau)})\to r_i(\Pi\bm{q},\theta)
	\quad\mbox{strongly in }L^\eta(\Omega_T),\ \eta<\infty.
$$}

By assumption, $M_{ij}(\bm\rho^{(\tau)},\theta^{(\tau)})$ and
$M_j(\bm\rho^{(\tau)},\theta^{(\tau)})/\theta^{(\tau)}$ are bounded.
Then the strong convergences imply that these sequences are converging
in $L^q(\Omega_T)$ for $q<\infty$, and the limits can be identified. Thus,
\begin{align*}
  M_{ij}(\bm\rho^{(\tau)},\theta^{(\tau)})\to M_{ij}(\bm\rho,\theta)
	&\quad\mbox{strongly in }L^q(\Omega_T), \\
  M_j(\bm{\rho}^{(\tau)},\theta^{(\tau)})/\theta^{(\tau)}\to
	M_j(\bm\rho,\theta)/\theta
	&\quad\mbox{strongly in }L^q(\Omega_T)\mbox{ for all }q<\infty.
\end{align*}
This shows that
$$
  M_j(\bm{\rho}^{(\tau)},\theta^{(\tau)})
	= \frac{1}{\theta^{(\tau)}}M_j(\bm{\rho}^{(\tau)},\theta^{(\tau)})
	\theta^{(\tau)}
	\to \frac{1}{\theta}M_j(\bm\rho,\theta)\theta = M_j(\bm\rho,\theta)
$$
strongly in $L^\eta(\Omega_T)$ for $\eta<16/3$. Moreover, taking into account
\eqref{3.nath}, we have
$$
	M_j(\bm{\rho}^{(\tau)},\theta^{(\tau)})\na\frac{1}{\theta^{(\tau)}}
	= -\frac{M_j(\bm{\rho}^{(\tau)},\theta^{(\tau)})}{\theta^{(\tau)}}
	\na\log\theta^{(\tau)}
	\rightharpoonup \frac{M_j(\bm\rho,\theta)}{\theta}\na\log\theta
$$
weakly in $L^\eta(\Omega_T)$ for $\eta<8/3$. 
Finally, by the weak convergence of $(\na\bm{v}^{(\tau)})$ in $L^2(\Omega_T)$, 
\begin{align*}
  M_{ij}(\bm{\rho}^{(\tau)},\theta^{(\tau)})\na v_j^{(\tau)}
	\rightharpoonup M_{ij}(\bm\rho,\theta)\na v_j
	&\quad\mbox{weakly in }L^\eta(\Omega_T),\ \eta < 2, \\
  M_j(\bm{\rho}^{(\tau)},\theta^{(\tau)})\na v_j^{(\tau)}
	\rightharpoonup M_j(\bm\rho,\theta)\na v_j
	&\quad\mbox{weakly in }L^\eta(\Omega_T),\ \eta < 16/11, \\
	M_j(\bm{\rho}^{(\tau)},\theta^{(\tau)})\na\frac{1}{\theta^{(\tau)}}
	\rightharpoonup -\frac{1}{\theta^2}M_j(\bm\rho,\theta)\na\theta
	&\quad\mbox{weakly in }L^\eta(\Omega_T),\ \eta < 8/7.
\end{align*}

These convergences allow us to perform the limit $(\eps,\tau)\to 0$.
Finally, we can show as in \cite[p.~1980f]{Jue15} that the linear interpolant 
$\widetilde\rho_i^{(\tau)}$ of $\rho_i^{(\tau)}$ and the piecewise constant function 
$\rho_i^{(\tau)}$ converge
to the same limit, which leads to $\rho_i^0=\widetilde\rho_i^{(\tau)}(0)
\rightharpoonup\rho_i(0)$ weakly in $H^2(\Omega)'$. Thus, the initial
datum $\rho_i(0)=\rho_i^0$ is satisfied in the sense of $H^2(\Omega)'$.
Similarly, $(\rho\theta)(0)=\rho^0\theta^0$ in the sense of $W^{1,16}(\Omega)'$.
This finishes the proof.

%%%%%%%%%%%%%%%%%%%%%%%%%%%%%%%%%%%%%%%%%%%%%%%%%%%%%%%%%%%%%%%%%%%%%%%%%%%%%%%

\section{Proof of Theorem \ref{thm.ex2}}\label{sec.degen}

\red{The proof of Theorem \ref{thm.ex2} is very similar to that one from Section
\ref{sec.ex}, therefore we present only the changes in the proof. Steps 1--3 are the
same as in the previous section. Only the estimate of $I_4$ is different:
$$
  I_4 = \int_\Omega\sum_{i,j=1}^{n-1}M_{ij}\na v_i\cdot\na v_j dx
	= \int_\Omega\sum_{i,j=1}^nM_{ij}\na q_i\cdot\na q_j dx
  \ge \frac{c_M}{n}\int_\Omega\sum_{i=1}^n\rho_i|\na(\Pi\bm{q})_i|^2 dx.
$$
This gives a uniform estimate for 
$\int_\Omega\rho_i^{(\tau)}|\na(\Pi\bm{q}^{(\tau)})_i|^2 dx$. We claim that
it yields a bound for $\na(\rho_i^{(\tau)})^{1/2}$ in $L^2(\Omega_T)$. Indeed,
we insert the definitions $q_i^{(\tau)}=\log(\rho_i^{(\tau)}/\theta^{(\tau)})$ 
and $(\Pi\bm{q}^{(\tau)})_i = q_i^{(\tau)}-\sum_{j=1}^n q_j^{(\tau)}/n
= \log\rho_i^{(\tau)} - \sum_{j=1}^n(\log\rho_j^{(\tau)})/n$ to find that
\begin{align*}
  \sum_{i=1}^n\rho_i|\na(\Pi\bm{q}^{(\tau)})_i|^2 
	&= \sum_{i=1}^n\rho_i^{(\tau)}\bigg|\na\log\rho_i^{(\tau)}
	- \frac{1}{n}\sum_{j=1}^n\na\log\rho_j^{(\tau)}\bigg|^2 \\
	&= \sum_{i=1}^n\rho_i^{(\tau)}|\na\log\rho_i^{(\tau)}|^2 
	- \frac{2}{n}\na\rho^0\cdot\sum_{j=1}^n\na\log\rho_j^{(\tau)} 
	+ \frac{\rho^0}{n^2}\bigg|\sum_{j=1}^n\na\log\rho_j^{(\tau)}\bigg|^2 \\
	&\ge 4\sum_{i=1}^n|\na(\rho_i^{(\tau)})^{1/2}|^2
	- 4|\na(\rho^0)^{1/2}|^2.
\end{align*}
This shows the claim.}

\red{In contrast to Step 4 in Section \ref{sec.ex}, we do not have a uniform
bound for $v_i^{(\tau)}$ in $L^2(0,T;$ $H^1(\Omega))$ but
a bound for $(\rho_i^{(\tau)})^{1/2}$. We deduce from the $L^\infty$ bound for 
$\rho_i^{(\tau)}$ a bound for $\rho_i^{(\tau)}$ in $L^2(0,T;H^1(\Omega))$,
using $\na\rho_i^{(\tau)}=(\rho_i^{(\tau)})^{1/2}\na(\rho_i^{(\tau)})^{1/2}$.
This bound changes the proof of estimate \eqref{est.time} for the time translates.
In fact, we just have to replace the estimations involving $\na v_j^{(\tau)}$:
\begin{align*}
  \int_0^T\int_\Omega &
	\bigg|\sum_{j=1}^{n-1} M_{ij}^{(\tau)}\na v_j^{(\tau)}\cdot\na\phi_j dxdt\bigg|dxdt
	= \int_0^T\int_\Omega \bigg|\sum_{j=1}^n M_{ij}^{(\tau)}\na\log\rho_j^{(\tau)}
	\cdot\na\phi_i\bigg| dxdt \\
	&\le \sum_{j=1}^n\|M_{ij}^{(\tau)}/\rho_j^{(\tau)}
	\|_{L^\infty(\Omega_T)}\|\na\rho_j^{(\tau)}\|_{L^2(\Omega_T)}
	\|\na\phi_i\|_{L^2(\Omega_T)}, \\
  \int_0^T\int_\Omega & \bigg|\sum_{j=1}^{n-1}M_j^{(\tau)}\na v_j^{(\tau)}
	\cdot\na\phi_0\bigg| dxdt
	= \int_0^T\int_\Omega\bigg|\sum_{j=1}^n M_j^{(\tau)}\na\log\rho_j^{(\tau)}
	\cdot\na\phi_0 \bigg|dxdt \\
	&\le \sum_{j=1}^n\|M_{ij}^{(\tau)}/\rho_j^{(\tau)}
	\|_{L^\infty(\Omega_T)}\|\na\rho_j^{(\tau)}\|_{L^2(\Omega_T)}
	\|\na\phi_0\|_{L^2(\Omega_T)}.
\end{align*}
This yields \eqref{est.time}.
}

\red{The $L^2(0,T;H^1(\Omega))$ estimate for $\rho_i^{(\tau)}$ and
\eqref{est.time} allow us to apply the Aubin--Lions lemma 
in the version of \cite{DrJu12}
yielding, up to a subsequence, the strong convergence $\rho_i^{(\tau)}\to
\rho_i$ in $L^2(\Omega_T)$ as $(\eps,\tau)\to 0$ and, 
because of the boundedness of $\rho_i^{(\tau)}$, in $L^r(\Omega_T)$ 
for any $r<\infty$.}

\red{It remains to perform the limit $(\eps,\tau)\to 0$ in the terms involving
$\bm{v}^{(\tau)}$,
$$
  \sum_{j=1}^n M_{ij}(\bm\rho^{(\tau)},\theta^{(\tau)})\na v_j^{(\tau)}, \quad
	\sum_{i=1}^n M_i(\bm\rho^{(\tau)},\theta^{(\tau)})\na v_i^{(\tau)}, \quad
	\eps \theta^{(\tau)}(D^2 v_j^{(\tau)} + v_j^{(\tau)}).
$$
The last term is easy to treat: The bound for $\sqrt{\eps}v_j^{(\tau)}$ 
in $L^2(0,T;H^2(\Omega))$ and the strong convergence of $\theta^{(\tau)}$
imply that $\eps \theta^{(\tau)}(D^2 v_j^{(\tau)} + v_j^{(\tau)})\to 0$
strongly in $L^2(\Omega_T)$.
Since $M_{ij}/\rho_j^{(\tau)}$ is bounded by assumption, we have
$M_{ij}(\bm\rho^{(\tau)},\theta^{(\tau)})/\rho_j^{(\tau)}$ $\to
M_{ij}(\bm\rho,\theta)/\rho_j$ strongly in $L^r(\Omega_T)$ for $r<\infty$.
Hence, using \eqref{1.M1} and the weak convergence of $(\na\rho_j^{(\tau)})$
in $L^2(\Omega_T)$,
$$
  \sum_{j=1}^{n-1} M_{ij}(\bm\rho^{(\tau)},\theta^{(\tau)})\na v_j^{(\tau)}
	= \sum_{j=1}^n \frac{M_{ij}(\bm\rho^{(\tau)},\theta^{(\tau)})}{\rho_j^{(\tau)}}
	\na\rho_j^{(\tau)} \rightharpoonup 
	\sum_{j=1}^n \frac{M_{ij}(\bm\rho,\theta)}{\rho_j}\na\rho_j
$$
weakly in $L^\eta(\Omega_T)$ for $\eta<2$.
Since $(M_{ij}/\rho_j^{(\tau)})\na\rho_j^{(\tau)}$ is bounded in $L^2(\Omega_T)$,
this convergence also holds in $L^2(\Omega_T)$. 
The limit in the second term
$\sum_{i=1}^n M_i(\bm\rho^{(\tau)},\theta^{(\tau)})\na v_i^{(\tau)}$ is performed
in an analogous way, leading to
$$
  \sum_{i=1}^n M_i(\bm\rho^{(\tau)},\theta^{(\tau)})\na v_i^{(\tau)}
	= \sum_{i=1}^n \frac{M_{i}(\bm\rho^{(\tau)},\theta^{(\tau)})}{\rho_i^{(\tau)}}
	\na\rho_i^{(\tau)}
	\rightharpoonup \sum_{i=1}^n \frac{M_{i}(\bm\rho,\theta)}{\rho_i}\na\rho_i
$$
weakly in $L^2(\Omega_T)$. This finishes the proof.}

%%%%%%%%%%%%%%%%%%%%%%%%%%%%%%%%%%%%%%%%%%%%%%%%%%%%%%%%%%%%%%%%%%%%%%%%%%%%%%%

\end{document}